\documentclass[11pt]{article}
\usepackage[a4paper,margin=1.0in]{geometry}

\usepackage[colorlinks,citecolor=magenta,linkcolor=black]{hyperref}
\pdfpagewidth=\paperwidth \pdfpageheight=\paperheight
\usepackage{amsfonts,amssymb,amsthm,amsmath,eucal,tabu,url}
\usepackage{pgf}
 \usepackage{array}
 \usepackage{tikz-cd}
 \usepackage{pstricks}
 \usepackage{pstricks-add}
 \usepackage{pgf,tikz}
 \usetikzlibrary{automata}
 \usetikzlibrary{arrows}
 \usepackage{indentfirst}
\usepackage{tabularx} 
\theoremstyle{plain}
\newtheorem{thm}{Theorem}[section]
\newtheorem{theorem}[thm]{Theorem}

\newtheorem{proposition}[thm]{Proposition}

\theoremstyle{definition}
\newtheorem{definition}[thm]{Definition}
\newtheorem{remark}[thm]{Remark}
\newtheorem{example}[thm]{Example}

\newtheorem{thevarthm}[thm]{\varthmname}

\newenvironment{varthm*}[1]{\trivlist\item[]{\bf #1.}\it}{\endtrivlist}

\def\keywordname{{\bfseries Keywords}}%
\def\keywords#1{\par\addvspace\medskipamount{\rightskip=0pt plus1cm
\def\and{\ifhmode\unskip\nobreak\fi\ $\cdot$
}\noindent\keywordname\enspace\ignorespaces#1\par}}
\def\subclassname{{\bfseries Mathematics Subject Classification
(2020)}\enspace}
\def\subclass#1{\par\addvspace\medskipamount{\rightskip=0pt plus1cm
\def\and{\ifhmode\unskip\nobreak\fi\ $\cdot$
}\noindent\subclassname\ignorespaces#1\par}}

\begin{document}
\title{On homological properties of some Cynk-Szemberg octic hyperplane arrangements}
\author{Marek Janasz and Piotr Pokora}
\date{\today}
\maketitle

\thispagestyle{empty}
\begin{abstract}
In this paper we study Cynk-Szemberg octic hyperplane arrangements from the perspective of homological properties of their derivation modules. In particular, we define the notion of the type of hyperplane arrangements that will be used in our characterization of rigid Cynk-Szemberg octic hyperplane arrangements. Moreover, we deliver a combinatorial non-freeness criterion for essential hyperplane arrangements in $\mathbb{C}^{4}$.
\subclass{14N20, 52C35, 32S22}
\end{abstract}
\section{Introduction}

Our paper is motivated by a longstanding open conjecture in the theory of hyperplane arrangements in complex projective spaces: Terao's freeness conjecture -- see, for example, \cite{Dimca, Dimca1, terao}. If $\mathcal{A}$ is an arrangement of hyperplanes in a complex projective space $\mathbb{P}^{n}$, then Terao conjecture predicts that the freeness of $\mathcal{A}$ is determined by the intersection poset $L(\mathcal{A})$ of $\mathcal{A}$, i.e., the set of all subspaces obtained by intersecting some of the hyperplanes of $\mathcal{A}$, partially ordered by reverse inclusion. Recall that if $\mathcal{A} \subset \mathbb{P}^{n}$ is a hyperplane arrangement and $Q \in S:=\mathbb{C}[x_{0}, \ldots , x_{n}]$ is a defining equation of $\mathcal{A}$, then the derivation module $D(\mathcal{A})$ of $\mathcal{A}$ is defined as
$$D(\mathcal{A}) = \{\theta \in {\rm Der}(S) \, : \, \theta(\alpha_{H}) \in S \cdot \alpha_{H} \,\text{ for all } \, H \in \mathcal{A}\},$$
where $\alpha_{H} \in S$ denotes a linear form such that ${\rm ker}(\alpha_{H}) = H$ for each $H \in \mathcal{A}$. An arrangement $\mathcal{A}$ is \textbf{free} if $D(\mathcal{A})$ is a free $S$-module. It is worth mentioning here that the class of free arrangements is still mysterious and we try to understand its basic properties, both combinatorial and geometric.
In the present note we want to present a systematic study of a certain class of hyperplane arrangements in $\mathbb{P}^{3}$, namely Cynk--Szemberg octic hyperplane arrangements that were introduced in \cite{CSz}. These hyperplane arrangements are interesting objects in algebraic geometry due to their applications, namely they can be used to construct Calabi-Yau threefolds as double covers of $\mathbb{P}^{3}$ branched along these octic hyperplane arrangements. From the combinatorial perspective of the intersection lattices, Cynk--Szemberg octics admit only double and triple intersection lines and arrangement $q$-fold points with $q\in \{2,3,4,5\}$, so these conditions are rather restrictive. Nevertheless, according to our understanding, these octic hyperplane arrangements have not yet been completely classified. Here we focus on homological properties of these arrangements, namely we want to detect these examples of Cynk--Szemberg octics that are close to be free, and this property will be measured by the so-called type of hyperplane arrangements, which is a direct generalization of the notion defined by the second author with Abe and Dimca in \cite{her} for reduced plane curves. Our main result provides a complete homological classification of rigid Cynk--Szemberg octics in the language of the type, see Theorem \ref{rigOC}. We present also a purely combinatorial non-freeness criterion, see Proposition \ref{gen}, that we will use in the setting of our Cynk--Szemberg octics. 

In the paper we work exclusively over the complex numbers.

\section{Preliminaries}
We start with a general preliminaries on hyperplane arrangements. Let $V = \mathbb{C}^{n}$ and consider $S = {\rm Sym}^{*}(V^{*}) \cong \mathbb{C}[x_{1}, \ldots , x_{n}]$. For $n=4$, we assume that our coordinates are $x, y, z, w$. 

Let $\mathcal{A}$ be a central arrangement of hyperplanes in $V$, i.e., a finite set of distinct linear hyperplanes in $V$. Since we consider only central hyperplane arrangements in $\mathbb{C}^{n}$, these objects can be treated as arrangements of hyperplanes in $\mathbb{P}^{n-1}$ via a natural identification, see \cite[Definition 2.4]{Dimca}. For each $H \in \mathcal{A}$ we fix a linear form $\alpha_{H}$ such that ${\rm ker}(\alpha_{H}) = H$. Denote by $Q(\mathcal{A}) = \prod_{H \in \mathcal{A}}\alpha_{H}$ the defining equation of $\mathcal{A}$, and let ${\rm Der}(S) =\bigoplus_{i=1}^{n} S\cdot \partial_{x_{i}}$.
The derivation module $D(\mathcal{A})$ of $\mathcal{A}$ is defined as
$$D(\mathcal{A}) = \{\theta \in {\rm Der}(S) \, : \, \theta(\alpha_{H}) \in S \cdot \alpha_{H} \,\text{ for all } \, H \in \mathcal{A}\}.$$
\begin{definition}
A central hyperplane arrangement $\mathcal{A} \subset \mathbb{C}^{n}$ is free with the exponents ${\rm exp}(A) = (d_{1}, \ldots ,d_{n})$ if there exists a homogeneous $S$-basis $\theta_{1}, \ldots , \theta_{n}$ for $D(\mathcal{A})$ such that ${\rm deg} (\theta_{i}) = d_{i}$ for all $i=1,\ldots,n$.
\end{definition}

\begin{definition}
	Let $\mathcal{A}$ be an arrangement in $\mathbb{C}^n$ which is not necessary free. We say that 
	$(d_1,\ldots,d_s)$ is a \textbf{degree sequence} of $\mathcal{A}$, denoted by $\exp(\mathcal{A})$,  if 
	the $0$-th syzygy of a minimal free resolution of $D(\mathcal{A})$ is of the form 
	$\oplus_{i=1}^s S[-d_i]$.
	\label{degreesequence}
\end{definition}
From now on we assume that our arrangements are always central and non-empty. Then every generating set of $D(\mathcal{A})$ contains at least one element of degree $1$, and we can choose this element as the well-known Euler derivation
$$\theta_{E} = \sum_{i=1}^{\ell}x_{i} \partial_{x_{i}}.$$
Let us define
$$D_{0}(\mathcal{A}) = \{\theta \in {\rm Der}(S) \, : \, \theta(\alpha_{H}) = 0 \, \text{ for all } \, H \in \mathcal{A}\}.$$
Then the following decomposition holds \cite[pp. 151 -- 152]{Dimca}:
$$D(\mathcal{A}) = S\cdot \theta_{E} \oplus D_{0}(\mathcal{A}).$$
\begin{remark}
We will use a well-known convention that by ${\rm exp}_{0}(\mathcal{A})$ we denote the degree sequence for $D_{0}(\mathcal{A})$. Moreover, for a free arrangement $\mathcal{A}$ we write ${\exp}(\mathcal{A}) = (1,{\rm exp}_{0}(\mathcal{A}))$.
\end{remark}
\begin{remark}
In light of the above decomposition, determining the freeness of $\mathcal{A}$ boils down to checking whether $D_{0}(\mathcal{A})$ is a free $S$-module.
\end{remark}
\noindent
Now we introduce the Poincar\'e polynomial for $\mathcal{A}$. Denote by $L(\mathcal{A})$ the intersection lattice of $\mathcal{A}$, i.e.,
$$L(\mathcal{A}) = \{ \cap_{H \in \mathcal{B}}H \, : \, \mathcal{B} \subseteq \mathcal{A}\},$$
which is ordered by reverse inclusion.  We define a rank function for elements in $L(\mathcal{A})$, namely for $X \in L(\mathcal{A})$ one has
$$r(X) = n - {\rm dim} \, X.$$
Recall now that the dimension of an arrangement ${\rm dim}(\mathcal{A})$ of $\mathcal{A}$ is defined to be as ${\rm dim}(\mathbb{C}^{n})=n$, and the rank of $\mathcal{A}$ ${\rm rank}(\mathcal{A})$ is defined as the dimension of the space spanned by the normals of the hyperplanes in $\mathcal{A}$. We say that an arrangement $\mathcal{A}$ is \textbf{essential} if ${\rm rank}(\mathcal{A}) = {\rm dim}(\mathcal{A})$.

Let $\mu : L(\mathcal{A}) \times L(\mathcal{A}) \rightarrow \mathbb{Z}$ be the M\"obius function of $L$ and for $X\in L(\mathcal{A})$  we define $\mu(X) : = \mu(V,X)$. A Poincar\'e polynomial of $\mathcal{A}$ is defined by
$$\pi(\mathcal{A};t) = \sum_{X \in L(\mathcal{A})}\mu(X)(-t)^{r(X)}.$$
It is well-known that the Poincar\'e polynomial is a degree $r(\mathcal{A})= \max_{X\in L(\mathcal{A})} r(X)$ polynomial in $t$ with non-negative coefficients, and for essential arrangements we have $r(\mathcal{A}) = {\rm rank}(\mathcal{A}) = n$. We will use the following fundamental result due to Terao \cite{terao}.
\begin{theorem}[Terao's factorization]
Assume that $\mathcal{A}$ is free with exponents ${\rm exp}(\mathcal{A}) = (d_{1}, \ldots, d_{n})$. Then
$$\pi(\mathcal{A};t) =\prod_{i}(1+d_{i}t),$$
which also means that for free arrangements their exponents are determined by the intersection lattice.
\end{theorem}

Finally, we are going to define the type of an essential hyperplane arrangement $\mathcal{A}\subset \mathbb{C}^{n}$.
\begin{definition}[Type of a hyperplane arrangement]
Let $\mathcal{A} \subset \mathbb{C}^{n}$ be a essential arrangement of $k$ hyperplanes with ${\rm exp}_{0}(\mathcal{A}) = (d_{1}, \ldots, d_{s})$ and $s\geq n-1$. Without loss of generality, we can assume that $1 \leq d_{1} \leq d_{2} \leq d_{3} \leq \ldots \leq d_{s}$.
Then the type of $\mathcal{A}$ is defined as
$$t(\mathcal{A}) = \sum_{i=1}^{n-1} d_{i} - k + 1.$$
\end{definition}
We should mention that the notion of type has been recently introduced in \cite{her} in the setting of reduced plane curves in the complex projective plane. At first, this notion may seem a bit strange, but let us present some observations that justify its importance.
\begin{example}[Free hyperplane arrangements]
Let $\mathcal{A} \subset \mathbb{C}^{4}$ be an essential free arrangement of $k$ hyperplanes with ${\rm exp}_{0}(\mathcal{A}) = (d_{1}, d_{2}, d_{3}) $. According to Terao's factorization result, we know that the coefficient of the linear term of $\pi(\mathcal{A};t)$ is equal to the number of hyperplanes $k$, and simultaneous it is also equal to $1+ d_{1}+d_{2}+d_{3}$, hence 
$$d_{1}+d_{2}+d_{3}=k-1.$$
This gives us that for essential free hyperplane arrangements in $\mathbb{C}^{4}$ we have $t(\mathcal{A}) = 0$. On the other hand, if $t(\mathcal{A}) = 0$, then $d_{1}+d_{2}+d_{3}=k-1$, and \cite[Lemma 4.12]{DS} implies that $\mathcal{A}$ is free.

Based on the above discussion, we can conclude that an essential hyperplane arrangement $\mathcal{A} \subset \mathbb{C}^{4}$ is free \textbf{if and only if} $t(\mathcal{A})=0$.
\end{example}
\begin{example}[{Nearly free hyperplane arrangements, \cite[Definition 5.3]{DS}}]
Here we would like to discuss the case of essential nearly free hyperplane arrangements in $\mathbb{C}^{4}$.
Let us recall that $\mathcal{A} : f= 0$ is a central hyperplane arrangement in $\mathbb{C}^{4}$, then we define the following $S$-module of algebraic relations associated with the Jacobian ideal $J_{f} = \langle \partial_{x}f, \partial_{y}f, \partial_{z}f , \partial_{w}f \rangle$, namely
$${\rm AR}(f) = \{ (a_{1}, \ldots ,a_4) \in S^{\oplus 4} : a_{1}\partial_{x}f + a_{2}\partial_{y}f + a_{3}\partial_{z}f +a_{4}\partial_{w}f=0\}.$$
It turns our that ${\rm AR}(f)$ is isomorphic with $D_{0}(\mathcal{A})$, see \cite[pp. 151 -- 152]{Dimca}.
We say that $\mathcal{A}$ is nearly free if the module ${\rm AR}(f)$ is not free of rank $3$, but it has exactly $4$ generators $r_{1},r_{2},r_{3},r_{4}$ such that ${\rm deg}(r_{i}) = d_{i}$ with $d_{3}=d_{4}$, and the second order syzygies (which exist since the module is not free) are spanned by a unique relation of the form
\[R : a_{1}r_{1} + a_{2}r_{2} + a_{3}r_{3} + a_{4}r_{4} = 0,\]
where $a_{1}, \ldots, a_{4}$ are homogeneous polynomials of degrees $d_{3}-d_{1}+1$, $d_{3}-d_{2}+1$, $1$, $1$, respectively. It turns out that if $\mathcal{A}$ is an essential arrangement of $k$ hyperplanes that is nearly free, then by \cite[Theorem 5.4 ii)]{DS} we know that ${\rm AR}(f)$ is generated in degrees $d_{1} \leq d_{2} \leq d_{3}=d_{4}$ such that $d_{1} + d_{2} + d_{3} = k$, and hence
$t(\mathcal{A})=1$.
\end{example}
\section{Cynk--Szemberg octic hyperplane arrangements}
Here we would like to introduce the notion of Cynk--Szemberg octic arrangements. The definition that we present below comes from Meyer's monograph \cite[Section 4.1]{Meyer} and it is a general definition, i.e., it does not focus just on hyperplanes.
\begin{definition}
Let $D \subset \mathbb{P}^{3}$ be a surface. We call $D$ as an arrangement if it is a sum of irreducible surfaces $D_{1}, \ldots , D_{r}$ with only isolated singular points satisfying the following conditions:
\begin{enumerate}
    \item for any $i\neq j$ the surfaces $D_{i}$ and $D_{j}$ intersect transversally along a smooth irreducible curve $C_{i,j}$ or they are disjoint;
    \item the curves $C_{i,j}$ and $C_{k,l}$ either coincide, are disjoint or intersect transversally.
\end{enumerate}
A singular point of $D_{i}$ is called as an isolated singular point of the arrangement. A point $p \in D$ which belongs to $p$ of the surfaces $D_{1}, \ldots, D_{r}$ is called as an arrangement $p$-fold point. We say that an irreducible curve $C \subset D$ is a $q$-fold curve if exactly $q$ of the surfaces $D_{1}, \ldots, D_{r}$ pass through it.
\end{definition}
We will use the following numerical data for a given arrangement $D$:
\begin{enumerate}
\item[a)] $k_{i}$ = the degree of $D_{i}$,
\item[b)] $t_{p}(1)$ = the number of $p$-fold lines,
\item[c)] $t_{q}$ = the number of arrangement $q$-fold points.
\end{enumerate}
\begin{definition}
If $D \subset \mathbb{P}^{3}$ is an arrangement of degree $8$, i.e., $k_{1} + \cdots + k_{r} =8$, then we call it as an octic arrangement.
\end{definition}
\begin{definition}[{\cite[Section 1]{Cynks}}]
If $D \subset \mathbb{P}^{3}$ is an arrangement consisting of eight hyperplanes, i.e., $D=\{H_{1}, \ldots, H_{8}\} \subset \mathbb{P}^{3}$ such that is has only double and triple lines and arrangement $q$-fold points with $q \in \{2,3,4,5\}$, then $D$ is called as a Cynk--Szemberg octic hyperplane arrangement.
\end{definition}
\begin{remark}
As we have mentioned in Introduction, Cynk--Szemberg octic hyperplane arrangements $D$ are very special since the double covers branched along $D$ have non-singular models $\widetilde{X}$ which are Calabi-Yau threefold. We refer to \cite{Cynks, Meyer} for a comprehensive discussion devoted to this interesting subject.
\end{remark}
\begin{remark}
In the setting of Cynk octic hyperplane arrangement, the following naive combinatorial counts hold:
\begin{equation}
\sum_{q\geq 3}\binom{q}{3}t_{q}  = 56 \quad \text{ and }  \quad \sum_{p\geq 2}t_{p}(1) = 28.
\end{equation}
\end{remark}
It is reasonable to expect that Cynk--Szemberg octic hyperplane arrangements are fully classified. Based on our literature review, we can conclude that there are $14$ examples of rigid octics \cite{Cynks}, i.e., those with $0$-dimensional moduli, and $63$ examples of one-parameter families described explicitly by Meyer \cite{Meyer}, where the author provides the defining equations of these families. Furthermore, Cynk and Kocel-Cynk recently performed a complete classification of octic arrangements, modulo the action of the group of projective transformations, using a combinatorial data structure called an \textit{incidence table}, as it was presented in \cite{Cynks}.
In our paper we focus only on rigid examples of Cynk--Szemberg octic, and at the end we provide a comment devoted to one-parameter families of Cynk--Szemberg octics.
\section{A combinatorial non-freeness criterion for essential hyperplane arrangements}
We start with a general results devoted to hyperplane arrangements with only double lines. This result should be viewed as a straightforward generalization of a plane result which says that an arrangement of $d\geq 3$ lines with only double points is free if and only if $d=3$. 
\begin{proposition}
\label{gen}
Let $\mathcal{A} \subset \mathbb{C}^{4}$ be an essential arrangement of $k\geq 4$ hyperplanes such it has only double intersection lines, i.e., $t_{2}(1) =\binom{k}{2}$ and $t_{r}(1)=0$ for $r>2$. If $\mathcal{A}$ is free then $k=4$.
\end{proposition}
\begin{proof}
Assume that $\mathcal{A}$ is free with ${\rm exp}_{0}(\mathcal{A}) = (d_{1},d_{2},d_{3})$. Without loss of generality we can assume that $1 \leq d_{1} \leq d_{2} \leq d_{3}$. Recall that the Terao factorization theorem implies the Poincar\'e polynomial of $\mathcal{A}$ splits over the rationals, namely
$$\pi(\mathcal{A};t) = (1+t)(1+d_{1}t)(1+d_{2}t)(1+d_{3}t),$$
which gives us the following system of equations:
$$\begin{cases}
d_{1}d_{2}+d_{1}d_{3}+d_{2}d_{3}+d_{1}+d_{2}+d_{3} = \binom{k}{2}\\
d_{1}+d_{2}+d_{3}+1=k.
\end{cases}$$
This gives
$$d_{1}(d_{1}-1) + d_{2}(d_{2}-1) + d_{3}(d_{3}-1)=0.$$
Recall that $d_{1}+d_{2}+d_{3}=k-1\geq 3$ and $d_{1},d_{2},d_{3}$ are positive integers, which implies that the only possible solution of the above equation is $d_{1}=d_{2}=d_{3}=1$, and hence $k=4$.
\end{proof}
\begin{remark}
Right now we want to discuss the freeness of essential arrangements of $k=4$ hyperplanes in $\mathbb{C}^{4}$ having only $6$ double lines. There are two cases:
\begin{itemize}
\item $t_{2}(1)=6$ and $t_{3}=4$. Then, up to projective equivalence, the arrangement $\mathcal{A}$ is given by $Q(x, y, z, w) = xyzw$ -- obviously $\mathcal{A}$ is free with ${\rm exp}_{0}(\mathcal{A}) = (1,1,1)$;
\item $t_{2}(1)=6$ and $t_{4}=1$. However, our arrangement cannot be essential, since the rank of the unique maximal element of the poset is not equal to the dimension.
\end{itemize}
These observations give us a complete picture of free essential hyperplane arrangements in $\mathbb{C}^{4}$ with only double intersection lines.
\end{remark}
Clearly, the fact that a hyperplane arrangement has only double lines does not imply that it has only triple points as intersections, as shown by the above example that comes from \cite{Cynks}.
\begin{example}
We consider family number \textnumero 266, which is defined by the following equation:
\[Q_{A,B}^{266}: \quad xyzw(y-2z+2w)(2x+y+2w)(Ax+By+Az)(Ax+(A+B)y-Az+Bw) = 0,\]
where $(A:B) \in \mathbb{P}^{1}$ subject to the following conditions: $(A:B) \not\in \{(1:0),(0:1)\}$, $A+4B\neq 0$, $A-2B\neq 0$, $A+2B\neq 0$ and finally $A+B\neq 0$.
Now we present the intersection poset for arrangements in this family. Fix the ground set $E= \{1, \ldots , 8\}$ giving labels of hyperplanes. Then we have the following incidences:
\begin{center}
\begin{tabular}{ c c c c c c c }
\multicolumn{7}{c}{$p$-fold lines} \\\hline \hline
 (1,2) & (1,3) & (1,4) & (1,5) & (1,6) & (1,7) & (1,8) \\
 (2,3) & (2,4) & (2,5) & (2,6) & (2,7) & (2,8) & (3,4) \\
 (3,5) & (3,6) & (3,7) & (3,8) & (4,5) & (4,6) & (4,7) \\
 (4,8) & (5,6) & (5,7) & (5,8) & (6,7) & (6,8) & (7,8) \\ 
\end{tabular}
\end{center}
and 
\begin{center}
\begin{tabular}{ c c c c c c c c }
\multicolumn{8}{c}{arrangement $q$-fold points} \\ \hline \hline 
(1,2,3,7) & (1,2,4,6) & (1,2,5,8) & (1,3,4) & (1,3,5,6) & (1,3,8) & (1,4,5) & (1,4,7) \\
 (1,4,8)  &  (1,5,7)  &  (1,6,7) &  (1,6,8) &  (1,7,8) &  (2,3,4,5) &  (2,3,6,8) &  (2,4,7) \\
 (2,4,8)  & (2,5,6,7) &  (2,7,8) &  (3,4,6) &  (3,4,7) &  (3,4,8) &  (3,5,7)  &  (3,5,8) \\
 (3,6,7)  &  (3,7,8)  &  (4,5,6) &  (4,5,7,8) &  (4,6,7) & (4,6,8) &  (5,6,8) & (6,7,8)
\end{tabular}.
\end{center}
\end{example}

\section{Rigid Cynk--Szemberg octic hyperplane arrangements}
In this section, we will use the machinery introduced in previous sections to describe the homological properties of the rigid Cynk--Szemberg octic hyperplane arrangements. These arrangements have trivial parameter spaces and they are unique up to projective equivalence. There are exactly $14$ such arrangements \cite{Cynks}. We will list these rigid arrangements coherently alongside Cynk and Kocel-Cynk list \cite{Cynks}. Our approach is systematic and provides all the necessary combinatorial details. All symbolic computations are performed using \texttt{SINGULAR} \cite{Singular}. \\\\
%--------------------------A1-----------------------------------------------------------------------------------------------------------------------
\noindent
\framebox[1.1\width]{Arrangement \textnumero 1.}
\vskip 5pt
\noindent
Let us denote this arrangement by $\mathcal{A}_{1} \subset \mathbb{P}^{3}$. It is given by the following defining equation:
$$Q_{1}(x,y,z,w) = xyzw(x+y)(y+z)(z+w)(w+x)=0.$$
Now we describe the intersection poset of $\mathcal{A}_{1}$. Fix the ground set $E= \{1, \ldots , 8\}$ giving labels of hyperplanes. Then we have the following incidences:
\begin{center}
\begin{tabular}{ c c c c c }
\multicolumn{5}{c}{$p$-fold lines} \\\hline \hline
 (1,2,5) & (1,3) & (1,4,8) & (1,6) & (1,7) \\
 (2,3,6) & (2,4) &  (2,7) & (2,8) & (3,4,7) \\
 (3,5) & (3,8) & (4,5) & (4,6) & (5,6) \\
 (5,7) & (5,8) & (6,7) & (6,8) & (7,8)
\end{tabular}
\end{center}
and 
\begin{center}
\begin{tabular}{ c c c c c c c }
\multicolumn{7}{c}{arrangement $q$-fold points} \\ \hline \hline 
(1,2,3,5,6) &  (1,2,4,5,8) & (1,2,5,7) & (1,3,4,7,8) & (1,4,6,8) &  (1,6,7) &  (2,3,4,6,7) \\
(2,3,6,8) &  (2,7,8) &  (3,4,5,7) &  & (3,5,8) &  (4,5,6) &  (5,6,7,8) 
\end{tabular}.
\end{center}
Have the data collected above, we can easily compute the Poincar\'e polynomial of $\mathcal{A}_{1}$, namely
$$\pi(\mathcal{A}_{1};t) = 1 + 8t + 24t^{2} + 31t^{3} + 14t^{4}.$$
Observe that $\pi(\mathcal{A}_{1};t)$ does not split over the rationals and hence $\mathcal{A}_{1}$ cannot be free. Using \texttt{SINGULAR} we can compute the degree sequence for $D_{0}(\mathcal{A}_{1})$, namely
$${\rm exp}_{0}(\mathcal{A}_{1}) = (2,3,3,3),$$
so its type is $t(\mathcal{A}_{1})=1$.
Finally, in order to show that $\mathcal{A}_{1}$ is nearly free, by \cite[Theorem 5.4]{DS} we have to check that the Jacobian ideal $J_{Q_{1}}$ is saturated, which can be verified by \texttt{SINGULAR}, hence $\mathcal{A}_{1}$ is nearly free.
%------------------------------A3------------------------------------------------------------------------------------------------------------------
\\\\
\noindent
\framebox[1.1\width]{Arrangement \textnumero 3.}
\vskip 5pt
\noindent
Let us denote this arrangement by $\mathcal{A}_{3} \subset \mathbb{P}^{3}$. It is given by the following defining equation:
$$Q_{3}(x,y,z,w) = xyzw(x+y)(y+z)(y-w)(x-y-z+w)=0.$$
We describe the intersection poset of $\mathcal{A}_{3}$. Fix the ground set $E= \{1, \ldots , 8\}$ giving labels of hyperplanes. Then we have the following incidences:
\begin{center}
\begin{tabular}{ c c c c c }
\multicolumn{5}{c}{$p$-fold lines} \\\hline \hline
 (1,2,5) & (1,3) & (1,4) & (1,6) & (1,7) \\
 (1,8) & (2,3,6) & (2,4,7) & (2,8) & (3,4) \\
 (3,5) & (3,7) & (3,8) & (4,5) & (4,6) \\
 (4,8) & (5,6) & (5,7) & (5,8) & (6,7) \\
       & (6,8) &       & (7,8) &

\end{tabular}
\end{center}
and 
\begin{center}
\begin{tabular}{ c c c c c }
\multicolumn{5}{c}{arrangement $q$-fold points} \\ \hline \hline 
 (1,2,3,5,6) &  (1,2,4,5,7) &  (1,2,5,8) &  (1,3,4) &  (1,3,7,8) \\
 (1,4,6,8) & (1,6,7) & (2,3,4,6,7) & (2,3,6,8) & (2,4,7,8) \\
 (3,4,5) & (3,4,8) & (3,5,7) & (3,5,8) & (4,5,6) \\
 & (4,5,8) & & (5,6,7,8) &

\end{tabular}.
\end{center}
We can compute the Poincar\'e polynomial of $\mathcal{A}_{3}$, namely
$$\pi(\mathcal{A}_{3};t) = 1 + 8t + 25t^{2} + 35t^{3} + 17t^{4}.$$
Observe that $\pi(\mathcal{A}_{3};t)$ does not split over the rationals and hence $\mathcal{A}_{3}$ cannot be free. We can compute the degree sequence for $D_{0}(\mathcal{A}_{3})$, namely
$${\rm exp}_{0}(\mathcal{A}_{3}) = (3,3,3,3,3),$$
hence its type is equal to
$$t(\mathcal{A}_{3})=3+3+3-8+1 = 2.$$
\\
%------------------------------------------A19------------------------------------------------------------------------------------------------------
\noindent
\framebox[1.1\width]{Arrangement \textnumero 19.}
\vskip 5pt
\noindent
Let us denote this arrangement by $\mathcal{A}_{19} \subset \mathbb{P}^{3}$. It is given by the following defining equation:
$$Q_{19}(x,y,z,w) = xyzw(x+y)(y+z)(x-z-w)(x+y+z-w)=0.$$
We describe the intersection poset of $\mathcal{A}_{19}$. Fix the ground set $E= \{1, \ldots , 8\}$ giving labels of hyperplanes. Then we have the following incidences:
\begin{center}
\begin{tabular}{ c c c c c c }
\multicolumn{6}{c}{$p$-fold lines} \\\hline \hline
 (1,2,5) & (1,3) & (1,4) & (1,6) & (1,7) & (1,8) \\
 (2,3,6) & (2,4) & (2,7) & (2,8) & (3,4) & (3,5) \\
 (3,7) & (3,8) & (4,5) & (4,6) & (4,7) & (4,8) \\
 (5,6) & (5,7) & (5,8) & (6,7) & (6,8) & (7,8)
\end{tabular}
\end{center}
and 
\begin{center}
\begin{tabular}{ c c c c c }
\multicolumn{5}{c}{arrangement $q$-fold points} \\ \hline \hline 
(1,2,3,5,6) & (1,2,4,5) & (1,2,5,7) & (1,2,5,8) & (1,3,4,7) \\
(1,3,8) & (1,4,6,8) & (1,6,7) & (1,7,8) & (2,3,4,6) \\
(2,3,6,7,8) &  (2,4,7) & (2,4,8) & (3,4,5,8) & (3,5,7)\\
(4,5,6,7) & (4,7,8) &  & (5,6,8) &  (5,7,8)
\end{tabular}.
\end{center}
We can compute the Poincar\'e polynomial of $\mathcal{A}_{19}$, namely
$$\pi(\mathcal{A}_{19};t) = 1 + 8t + 26t^{2} + 38t^{3} + 19t^{4}.$$
Observe that $\pi(\mathcal{A}_{19};t)$ does not split over the rationals and hence $\mathcal{A}_{19}$ cannot be free. We can compute the degree sequence for $D_{0}(\mathcal{A}_{19})$, namely
$${\rm exp}_{0}(\mathcal{A}_{19}) = (3,3,3,3),$$
hence its type is equal to
$$t(\mathcal{A}_{19})=3+3+3-8+1 = 2.$$
%----------------------------------------------------------A32----------------------------------------------------------------
\\
\noindent
\framebox[1.1\width]{Arrangement \textnumero 32.}
\vskip 5pt
\noindent
Let us denote this arrangement by $\mathcal{A}_{32} \subset \mathbb{P}^{3}$. It is given by the following defining equation:
$$Q_{32}(x,y,z,w) = xyzw(x+y)(y+z)(x-y-z-w)(x+y-z+w)=0.$$
We describe the intersection poset of $\mathcal{A}_{32}$. Fix the ground set $E= \{1, \ldots , 8\}$ giving labels of hyperplanes. Then we have the following incidences:
\begin{center}
\begin{tabular}{ c c c c c c }
\multicolumn{6}{c}{$p$-fold lines} \\\hline \hline
(1,2,5) & (1,3) & (1,4) & (1,6) & (1,7) & (1,8) \\
(2,3,6) & (2,4) & (2,7) & (2,8) & (3,4) & (3,5) \\
(3,7) & (3,8) & (4,5) & (4,6) & (4,7) & (4,8) \\
(5,6) & (5,7) & (5,8) & (6,7) & (6,8) & (7,8)

\end{tabular}
\end{center}
and 
\begin{center}
\begin{tabular}{ c c c c c }
\multicolumn{5}{c}{arrangement $q$-fold points} \\ \hline \hline 
(1,2,3,5,6) &  (1,2,4,5) & (1,2,5,7) &  (1,2,5,8) &  (1,3,4) \\
(1,3,7,8) & (1,4,6,7) & (1,4,8) & (1,6,8) & (2,3,4,6) \\  
(2,3,6,7) & (2,3,6,8) & (2,4,7,8) & (3,4,5,8) & (3,4,7) \\
(3,5,7) & (4,5,6) & (4,5,7) & (4,6,8) & (5,6,7,8)

\end{tabular}.
\end{center}
We can compute the Poincar\'e polynomial of $\mathcal{A}_{32}$, namely
$$\pi(\mathcal{A}_{32};t) = 1 + 8t + 26t^{2} + 39t^{3} + 20t^{4}.$$
Observe that $\pi(\mathcal{A}_{32};t)$ does not split over the rationals and hence $\mathcal{A}_{32}$ cannot be free. We can compute the degree sequence for $D_{0}(\mathcal{A}_{32})$, namely
$${\rm exp}_{0}(\mathcal{A}_{32}) = (3,3,3,4,4,4),$$
hence its type is equal to
$$t(\mathcal{A}_{32})=3+3+3-8+1 = 2.$$ 
\\
%----------------------------------------------------------A69---------------------------------------------------------------

\noindent
\framebox[1.1\width]{Arrangement \textnumero 69.}
\vskip 5pt
\noindent
Let us denote this arrangement by $\mathcal{A}_{69} \subset \mathbb{P}^{3}$. It is given by the following defining equation:
$$Q_{69}(x,y,z,w) = xyzw(x+y)(x-y+z)(x-y-w)(x+y-z-w)=0.$$
We describe the intersection poset of $\mathcal{A}_{69}$. Fix the ground set $E= \{1, \ldots , 8\}$ giving labels of hyperplanes. Then we have the following incidences:
\begin{center}
\begin{tabular}{ c c c c c c }
\multicolumn{6}{c}{$p$-fold lines} \\\hline \hline
(1,2,5) & (1,3) & (1,4) & (1,6) & (1,7) & (1,8) \\
(2,3) & (2,4) & (2,6) & (2,7) & (2,8) & (3,4) \\
(3,5) & (3,6) & (3,7) & (3,8) & (4,5) & (4,6) \\
(4,7) & (4,8) & (5,6) & (5,7) & (5,8) & (6,7) \\
      &       & (6,8) & (7,8) &       &
\end{tabular}
\end{center}
and 
\begin{center}
\begin{tabular}{ c c c c c }
\multicolumn{5}{c}{arrangement $q$-fold points} \\ \hline \hline 
 (1,2,3,5,6) & (1,2,4,5,7) &  (1,2,5,8) &  (1,3,4) & (1,3,7) \\  (1,3,8) & (1,4,6,8) & (1,6,7) & (1,7,8) & (2,3,4) \\
 (2,3,7,8) & (2,4,6) & (2,4,8) & (2,6,7) & (2,6,8) \\
 (3,4,5,8) & (3,4,6,7) & (3,5,7) & (3,6,8) & (4,5,6) \\
           & (4,7,8) &  & (5,6,7,8) &
\end{tabular}.
\end{center}
We can compute the Poincar\'e polynomial of $\mathcal{A}_{69}$, namely
$$\pi(\mathcal{A}_{69};t) = 1 + 8t + 27t^{2} + 41t^{3} + 21t^{4}.$$
Observe that $\pi(\mathcal{A}_{69};t)$ does not split over the rationals and hence $\mathcal{A}_{69}$ cannot be free. We can compute the degree sequence for $D_{0}(\mathcal{A}_{69})$, namely
$${\rm exp}_{0}(\mathcal{A}_{69}) = (3,3,3,4,4),$$
hence its type is equal to
$$t(\mathcal{A}_{69})=3+3+3-8+1 = 2.$$

%----------------------------------------------------------A93---------------------------------------------------------------

\noindent
\framebox[1.1\width]{Arrangement \textnumero 93.}
\vskip 5pt
\noindent
Let us denote this arrangement by $\mathcal{A}_{93} \subset \mathbb{P}^{3}$. It is given by the following defining equation:
$$Q_{93}(x,y,z,w) = xyzw(x+y)(x-y+z)(y-z-w)(x+z-w)=0.$$
We describe the intersection poset of $\mathcal{A}_{93}$. Fix the ground set $E= \{1, \ldots , 8\}$ giving labels of hyperplanes. Then we have the following incidences:
\begin{center}
\begin{tabular}{ c c c c c c }
\multicolumn{6}{c}{$p$-fold lines} \\\hline \hline
(1,2,5) &  (1,3) &  (1,4) &  (1,6) &  (1,7) &  (1,8)\\
(2,3) &  (2,4) &  (2,6) &  (2,7) &  (2,8) &  (3,4) \\  
(3,5) &  (3,6) &  (3,7) &  (3,8) &  (4,5) &  (4,6) \\  
(4,7) &  (4,8) &  (5,6) &  (5,7) &  (5,8) &  (6,7)\\  
& &  (6,8) & (7,8) & & \\
\end{tabular}
\end{center}
and 
\begin{center}
\begin{tabular}{ c c c c c }
\multicolumn{5}{c}{arrangement $q$-fold points} \\ \hline \hline 
 (1,2,3,5,6) &  (1,2,4,5) &  (1,2,5,7) &  (1,2,5,8) &  (1,3,4,8) \\
 (1,3,7) &  (1,4,6,7) &  (1,6,8) &  (1,7,8) &  (2,3,4,7)\\  
 (2,3,8) &  (2,4,6,8) &  (2,6,7) &  (2,7,8) &  (3,4,5)\\ 
 (3,4,6) &  (3,5,7) &  (3,5,8) &  (3,6,7,8) &  (4,5,6)\\ 
 & (4,5,7,8) &  (5,6,7) &  (5,6,8) & \\
\end{tabular}.
\end{center}
We can compute the Poincar\'e polynomial of $\mathcal{A}_{93}$, namely
$$\pi(\mathcal{A}_{93};t) = 1 + 8t + 27t^{2} + 42t^{3} + 22t^{4}.$$
Observe that $\pi(\mathcal{A}_{93};t)$ does not split over the rationals and hence $\mathcal{A}_{93}$ cannot be free. We can compute the degree sequence for $D_{0}(\mathcal{A}_{93})$, namely
$${\rm exp}_{0}(\mathcal{A}_{93}) = (3,3,4,4,4,4,4),$$
hence its type is equal to
$$t(\mathcal{A}_{93})=3+3+4-8+1 = 3.$$

%----------------------------------------------------------A238---------------------------------------------------------------
\noindent
\framebox[1.1\width]{Arrangement \textnumero 238.}
\vskip 5pt
\noindent
Let us denote this arrangement by $\mathcal{A}_{283} \subset \mathbb{P}^{3}$. It is given by the following defining equation:
$$Q_{238}(x,y,z,w) = xyzw(x+y+z-w)(x+y-z+w)(x-y+z+w)(-x+y+z+w)=0.$$
We describe the intersection poset of $\mathcal{A}_{238}$. Fix the ground set $E = \{1, \ldots , 8\}$ giving labels of hyperplanes. Then we have the following incidences:
\begin{center}
\begin{tabular}{ c c c c c c c }
\multicolumn{6}{c}{$p$-fold lines} \\\hline \hline
 (1,2) &  (1,3) &  (1,4) &  (1,5) &  (1,6) &  (1,7) &  (1,8)\\
 (2,3) &  (2,4) &  (2,5) &  (2,6) &  (2,7) &  (2,8) &  (3,4)\\
 (3,5) &  (3,6) &  (3,7) &  (3,8) &  (4,5) &  (4,6) &  (4,7)\\
 (4,8) &  (5,6) &  (5,7) &  (5,8) &  (6,7) &  (6,8) &  (7,8)\\
\end{tabular}
\end{center}
and 
\begin{center}
\begin{tabular}{ c c c c c }
\multicolumn{5}{c}{arrangement $q$-fold points} \\ \hline \hline 
 (1,2,3) & (1,2,4) & (1,2,5,6) &  (1,2,7,8) &  (1,3,4) \\
 (1,3,5,7) &  (1,3,6,8) &  (1,4,5,8) &  (1,4,6,7) &  (2,3,4)\\
 (2,3,5,8) &  (2,3,6,7) & (2,4,5,7) &  (2,4,6,8) &  (3,4,5,6)\\  (3,4,7,8) &  (5,6,7) &  (5,6,8) &  (5,7,8) &  (6,7,8) \\ 
\end{tabular}.
\end{center}
By Proposition \ref{gen}, the arrangement $\mathcal{A}_{238}$ cannot be free. We can compute the degree sequence for $D_{0}(\mathcal{A}_{238})$, namely
$${\rm exp}_{0}(\mathcal{A}_{238}) = (3,3,4,4,4,4),$$
hence its type is equal to
$$t(\mathcal{A}_{238})=3+3+4-8+1 = 3.$$

%----------------------------------------------------------A239---------------------------------------------------------------

\noindent
\framebox[1.1\width]{Arrangement \textnumero 239.}
\vskip 5pt
\noindent
Let us denote this arrangement by $\mathcal{A}_{239} \subset \mathbb{P}^{3}$. It is given by the following defining equation:
$$Q_{239}(x,y,z,w) = xyzw(x+y+z)(x+y+w)(x+z+w)(y+z+w)=0.$$
We describe the intersection poset of $\mathcal{A}_{239}$. Fix the ground set $E= \{1, \ldots , 8\}$ giving labels of hyperplanes. Then we have the following incidences:
\begin{center}
\begin{tabular}{ c c c c c c c }
\multicolumn{6}{c}{$p$-fold lines} \\\hline \hline
 (1,2) &  (1,3) &  (1,4) &  (1,5) &  (1,6) & 
 (1,7) &  (1,8) \\
 (2,3) &  (2,4) &  (2,5) &  (2,6) &  (2,7) & 
 (2,8) &  (3,4)\\  
 (3,5) & (3,6) &  (3,7) &  
 (3,8) &  (4,5) &  (4,6) &  (4,7)\\  (4,8) &
 (5,6) &  (5,7) &  (5,8) &  (6,7) &  (6,8) & 
 (7,8)
\end{tabular}
\end{center}
and 
\begin{center}
\begin{tabular}{ c c c c c }
\multicolumn{5}{c}{arrangement $q$-fold points} \\ \hline \hline 
(1,2,3,5) &  (1,2,4,6) & (1,2,7,8) &  (1,3,4,7) &  (1,3,6,8)\\
(1,4,5,8) &  (1,5,6) & (1,5,7) & (1,6,7) &  (2,3,4,8)\\
(2,3,6,7) & (2,4,5,7) & (2,5,6) & (2,5,8) &  (2,6,8)\\
(3,4,5,6) & (3,5,7) & (3,5,8) & (3,7,8) &  (4,6,7)\\
(4,6,8) &  (4,7,8) & (5,6,7) & (5,6,8) &  (5,7,8)\\
& & (6,7,8) & &
\end{tabular}.
\end{center}
By Proposition \ref{gen}, the arrangement $\mathcal{A}_{239}$ cannot be free. Using \texttt{SINGULAR} we can compute the degree sequence for $D_{0}(\mathcal{A}_{239})$, namely
$${\rm exp}_{0}(\mathcal{A}_{239}) = (4,4,4,4,4,4,4,4,4,4),$$
hence its type is equal to
$$t(\mathcal{A}_{239})= 4+4+4-8+1 = 5.$$

%----------------------------------------------------------A240---------------------------------------------------------------

\noindent
\framebox[1.1\width]{Arrangement \textnumero 240.}
\vskip 5pt
\noindent
Let us denote this arrangement by $\mathcal{A}_{240} \subset \mathbb{P}^{3}$. It is given by the following defining equation:
$$Q_{240}(x,y,z,w) = xyzw(x+y+z)(x+y-z+w)(x-y+z+w)(x-y-z-w)=0.$$
We describe the intersection poset of $\mathcal{A}_{240}$. Fix the ground set $E= \{1, \ldots , 8\}$ giving labels of hyperplanes. Then we have the following incidences:
\begin{center}
\begin{tabular}{ c c c c c c c }
\multicolumn{6}{c}{$p$-fold lines} \\\hline \hline
 (1,2) &  (1,3) &  (1,4) &  (1,5) &  (1,6) &  (1,7) &  (1,8)\\
 (2,3) &  (2,4) &  (2,5) &  (2,6) &  (2,7) &  (2,8) &  (3,4)\\
 (3,5) &  (3,6) & (3,7) & (3,8) & (4,5) &  (4,6) &  (4,7)\\
 (4,8) &  (5,6) &  (5,7) &  (5,8) & (6,7) &  (6,8) &  (7,8)
\end{tabular}
\end{center}
and 
\begin{center}
\begin{tabular}{ c c c c c }
\multicolumn{5}{c}{arrangement $q$-fold points} \\ \hline \hline 
(1,2,3,5) &  (1,2,4) &  (1,2,6) &  (1,2,7,8) 
&  (1,3,4) \\
(1,3,6,8) &  (1,3,7) &  (1,4,5,8) &  (1,4,6,7) &  (1,5,6) \\
(1,5,7) &  (2,3,4) & (2,3,6,7) & (2,3,8) &  (2,4,5,7) \\
(2,4,6,8) &  (2,5,6) &  (2,5,8) &  (3,4,5,6) &  (3,4,7,8) \\
(3,5,7) & (3,5,8) & (5,6,7) & (5,6,8) &  (5,7,8)\\
& & (6,7,8) & &
\end{tabular}.
\end{center}
By Proposition \ref{gen}, the arrangement $\mathcal{A}_{240}$ cannot be free. Using \texttt{SINGULAR} we can compute the degree sequence for $D_{0}(\mathcal{A}_{240})$, namely
$${\rm exp}_{0}(\mathcal{A}_{240}) = (4,4,4,4,4,4,4,4,4,4),$$
hence its type is equal to
$$t(\mathcal{A}_{240})= 4 + 4 + 4 - 8+1 = 5.$$

%----------------------------------------------------------A241---------------------------------------------------------------
\noindent
\framebox[1.1\width]{Arrangement \textnumero 241.}
\vskip 5pt
\noindent
Let us denote this arrangement by $\mathcal{A}_{241} \subset \mathbb{P}^{3}$. It is given by the following defining equation:
$$Q_{241}(x,y,z,w) = xyzw(x+y+z+w)(x+y-z-w)(y-z+w)(x+z-w)=0.$$
We describe the intersection poset of $\mathcal{A}_{241}$. Fix the ground set $E= \{1, \ldots , 8\}$ giving labels of hyperplanes. Then we have the following incidences:
\begin{center}
\begin{tabular}{ c c c c c c c }
\multicolumn{6}{c}{$p$-fold lines} \\\hline \hline
 (1,2) &  (1,3) &  (1,4) &  (1,5) &  (1,6) &  (1,7) &  (1,8)\\
 (2,3) &  (2,4) &  (2,5) &  (2,6) &  (2,7) &  (2,8) &  (3,4)\\
 (3,5) &  (3,6) &  (3,7) &  (3,8) &  (4,5) &  (4,6) &  (4,7)\\
 (4,8) &  (5,6) & (5,7) & (5,8) & (6,7) &  (6,8) &  (7,8)
\end{tabular}
\end{center}
and 
\begin{center}
\begin{tabular}{ c c c c c }
\multicolumn{5}{c}{arrangement $q$-fold points} \\ \hline \hline 
(1,2,3) &  (1,2,4) &  (1,2,5,6) &  (1,2,7,8) &  (1,3,4,8) \\
(1,3,5,7) & (1,3,6) & (1,4,5) & (1,4,6,7) &  (1,5,8)\\
(1,6,8) &  (2,3,4,7) & (2,3,5) &  (2,3,6,8) &  (2,4,5,8)\\
(2,4,6) & (2,5,7) & (2,6,7) & (3,4,5,6) &  (3,5,8) \\
(3,6,7) & (3,7,8) & (4,5,7) & (4,6,8) &  (4,7,8)\\
& & (5,6,7,8) & &
\end{tabular}.
\end{center}
By Proposition \ref{gen}, the arrangement $\mathcal{A}_{241}$ cannot be free. Using \texttt{SINGULAR} we can compute the degree sequence for $D_{0}(\mathcal{A}_{241})$, namely
$${\rm exp}_{0}(\mathcal{A}_{241}) = (3,4,4,4,4,4,4),$$
hence its type is equal to
$$t(\mathcal{A}_{241})= 3 + 4 + 4 - 8+1 = 4.$$

%----------------------------------------------------------A245---------------------------------------------------------------
\noindent
\framebox[1.1\width]{Arrangement \textnumero 245.}
\vskip 5pt
\noindent
Let us denote this arrangement by $\mathcal{A}_{245} \subset \mathbb{P}^{3}$. It is given by the following defining equation:
$$Q_{245}(x,y,z,w) = xyzw(x+y+z)(y+z+w)(x-y-w)(x-y+z+w)=0.$$
We describe the intersection poset of $\mathcal{A}_{245}$. Fix the ground set $E = \{1, \ldots , 8\}$ giving labels of hyperplanes. Then we have the following incidences:
\begin{center}
\begin{tabular}{ c c c c c c c }
\multicolumn{6}{c}{$p$-fold lines} \\\hline \hline
(1,2) &  (1,3) & (1,4) &  (1,5) & (1,6) &  (1,7) &  (1,8)\\
(2,3) &  (2,4) & (2,5) &  (2,6) &  (2,7) &  (2,8) &  (3,4)\\
(3,5) &  (3,6) & (3,7) &  (3,8) &  (4,5) &  (4,6) &  (4,7)\\
(4,8) &  (5,6) &  (5,7) & (5,8) & (6,7) &  (6,8) &  (7,8)
\end{tabular}
\end{center}
and 
\begin{center}
\begin{tabular}{ c c c c c c c }
\multicolumn{5}{c}{arrangement $q$-fold points} \\ \hline \hline 
 (1,2,3,5) &  (1,2,4,7) &  (1,2,6,8) &  (1,3,4) &  (1,3,6,7) & (1,3,8) & (1,4,5,6) \\
 (1,4,8) & (1,5,7) &  (1,5,8) & (1,7,8) & (2,3,4,6) & (2,3,7) & (2,3,8) \\
 (2,4,5,8) & (2,5,6,7) & (2,7,8) & (3,4,5) & (3,4,7,8) & (3,5,6) & (3,5,7) \\
 (3,5,8) & (3,6,8) & (4,5,7) &  (4,6,7) & (4,6,8)&  (5,6,8)& (5,7,8) \\
 & & & (6,7,8) & & &
\end{tabular}.
\end{center}
By Proposition \ref{gen}, the arrangement $\mathcal{A}_{245}$ cannot be free. Using \texttt{SINGULAR} we can compute the degree sequence for $D_{0}(\mathcal{A}_{245})$, namely
$${\rm exp}_{0}(\mathcal{A}_{245}) = (4,4,4,4,4,4,4,4,4),$$
hence its type is equal to
$$t(\mathcal{A}_{245})=  4 + 4  + 4 - 8+1 = 5.$$

%----------------------------------------------------------ArrA---------------------------------------------------------------
\noindent
\framebox[1.1\width]{Arrangement $\mathfrak{A}$.}
\vskip 5pt
\noindent
Let us denote this arrangement by $\mathcal{A}_{\mathfrak{A}} \subset \mathbb{P}^{3}$. It is given by the following defining equation:
$$Q_{\mathfrak{A}}(x,y,z,w) =xyzw(x+y)(x+y+z-w)((\sqrt{-3}-1)x-2y+(\sqrt{-3}-1)z)(2y+(-\sqrt{-3}+1)z-2w) = 0.$$
We describe the intersection poset of $\mathcal{A}_{\mathfrak{A}}$. Fix the ground set $E= \{1, \ldots , 8\}$ giving labels of hyperplanes. Then we have the following incidences:
\begin{center}
\begin{tabular}{ c c c c c c c }
\multicolumn{6}{c}{$p$-fold lines} \\\hline \hline
 (1,2,5) &  (1,3) & (1,4) & (1,6) & (1,7) & (1,8) \\
 (2,3) & (2,4) & (2,6) & (2,7) & (2,8) & (3,4) \\
 (3,5) & (3,6) & (3,7) & (3,8) & (4,5) & (4,6) \\
 (4,7) & (4,8) & (5,6) & (5,7) & (5,8) & (6,7) \\
       &       & (6,8) & (7,8) &       &
 
\end{tabular}
\end{center}
and 
\begin{center}
\begin{tabular}{ c c c c c }
\multicolumn{5}{c}{arrangement $q$-fold points} \\ \hline \hline 
(1,2,3,5,7) & (1,2,4,5) & (1,2,5,6) & (1,2,5,8) & (1,3,4) \\
(1,3,6,8) & (1,4,6) & (1,4,7,8) & (1,6,7) & (2,3,4,8) \\
(2,3,6) & (2,4,6,7) & (2,6,8) & (2,7,8) & (3,4,5,6) \\
(3,4,7) & (3,5,8) & (3,6,7) & (3,7,8) & (4,5,7) \\
& (4,5,8) & (4,6,8) & (5,6,7,8) &
\end{tabular}.
\end{center}
We can compute the Poincar\'e polynomial of $\mathcal{A}_{\mathfrak{A}}$, namely
$$\pi(\mathcal{A}_{\mathfrak{A}};t) = 1 + 8t + 27t^{2} + 42t^{3} + 22t^{4},$$ 
hence $\mathcal{A}_{\mathfrak{A}}$ cannot be free. Using \texttt{SINGULAR} we can compute the degree sequence for $D_{0}(\mathcal{A}_{\mathfrak{A}})$, namely
$${\rm exp}_{0}(\mathcal{A}_{\mathfrak{A}}) = (3,3,4,4,4,4,4),$$
hence its type is equal to
$$t(\mathcal{A}_{\mathfrak{A}})= 3+3+4-8+1=3.$$

%----------------------------------------------------------ArrB---------------------------------------------------------------
\noindent
\framebox[1.1\width]{Arrangement $\mathfrak{B}$.}
\vskip 5pt
\noindent
Let us denote this arrangement by $\mathcal{A}_{\mathfrak{B}} \subset \mathbb{P}^{3}$. It is given by the following defining equation:
\begin{multline*}
 Q_{\mathfrak{B}}(x,y,z,w) = xyzw(x+y+z)(x+z-w)((\sqrt{-3}-1)x+(\sqrt{-3}+1)y-2z+2w) \\
\times((\sqrt{-3}-1)x+(\sqrt{-3}-1)y-2z+(\sqrt{-3}+1)w)= 0.   
\end{multline*}
We describe the intersection poset of $\mathcal{A}_{\mathfrak{B}}$. Fix the ground set $E= \{1, \ldots , 8\}$ giving labels of hyperplanes. Then we have the following incidences:
\begin{center}
\begin{tabular}{ c c c c c c c }
\multicolumn{7}{c}{$p$-fold lines} \\\hline \hline
(1,2) & (1,3) & (1,4) & (1,5) & (1,6) & (1,7) & (1,8) \\
(2,3) & (2,4) & (2,5) & (2,6) & (2,7) & (2,8) & (3,4) \\
(3,5) & (3,6) & (3,7) & (3,8) & (4,5) & (4,6) & (4,7) \\
(4,8) & (5,6) & (5,7) & (5,8) & (6,7) & (6,8) & (7,8)

\end{tabular}
\end{center}
and 
\begin{center}
\begin{tabular}{ c c c c c c c }
\multicolumn{7}{c}{arrangement $q$-fold points} \\ \hline \hline 
 (1,2,3,5) & (1,2,4) & (1,2,6,7) & (1,2,8) & (1,3,4,6) & (1,3,7,8) & (1,4,5) \\
 (1,4,7) & (1,4,8) & (1,5,6,8) & (1,5,7) & (2,3,4) & (2,3,6) & (2,3,7) \\
 (2,3,8) & (2,4,5,6) &  (2,4,7,8) & (2,5,7) & (2,5,8) & (2,6,8) & (3,4,5,8) \\
 (3,4,7) & (3,5,6,7) & (3,6,8) & (4,5,7) & (4,6,7) & (4,6,8) & (5,7,8) \\
 & & & (6,7,8) & & &

\end{tabular}.
\end{center}
By Proposition \ref{gen}, the arrangement $\mathcal{A}_{\mathfrak{B}}$ cannot be free. Using \texttt{SINGULAR} we can compute the degree sequence for $D_{0}(\mathcal{A}_{\mathfrak{B}})$, namely
$${\rm exp}_{0}(\mathcal{A}_{\mathfrak{B}}) = (4,4,4,4,4,4,4,4,4),$$
hence its type is equal to
$$t(\mathcal{A}_{\mathfrak{B}})= 4+4+4-8+1=5.$$

%----------------------------------------------------------ArrB---------------------------------------------------------------
\noindent
\framebox[1.1\width]{Arrangement $\mathfrak{C}$.}
\vskip 5pt
\noindent
Let us denote this arrangement by $\mathcal{A}_{\mathfrak{C}} \subset \mathbb{P}^{3}$. It is given by the following defining equation:
\begin{multline*}
 Q_{\mathfrak{C}}(x,y,z,w) = xyzw(x+y+z)((\sqrt{5}-1)y-2z+2w) (2x+2y+(\sqrt{5}-1)w) \\
 \times ((-\sqrt{5}+3)x+2y+(-\sqrt{5}+1)z+(\sqrt{5}-1)w)= 0.   
\end{multline*}
We describe the intersection poset of $\mathcal{A}_{\mathfrak{C}}$. Fix the ground set $E= \{1, \ldots , 8\}$ giving labels of hyperplanes. Then we have the following incidences:
\begin{center}
\begin{tabular}{ c c c c c c c }
\multicolumn{7}{c}{$p$-fold lines} \\\hline \hline
 (1,2) & (1,3) & (1,4) & (1,5) & (1,6) & (1,7) & (1,8) \\
 (2,3) & (2,4) & (2,5) & (2,6) & (2,7) & (2,8) & (3,4) \\
 (3,5) & (3,6) & (3,7) & (3,8) & (4,5) & (4,6) & (4,7) \\
 (4,8) & (5,6) & (5,7) & (5,8) & (6,7) & (6,8) & (7,8)
\end{tabular}
\end{center}
and 
\begin{center}
\begin{tabular}{ c c c c c c c }
\multicolumn{7}{c}{arrangement $q$-fold points} \\ \hline \hline 
 (1,2,3,5) &  (1,2,4,7) & (1,2,6,8) & (1,3,4) & (1,3,6) & (1,3,7,8) & (1,4,5) \\
 (1,4,6) & (1,4,8) & (1,5,6,7) & (1,5,8) & (2,3,4,6) & (2,3,7) & (2,3,8) \\
 (2,4,5) & (2,4,8) & (2,5,6) & (2,5,7,8) & (2,6,7) & (3,4,5,7) & (3,4,8) \\
 (3,5,6) & (3,5,8) & (3,6,7) & (3,6,8) & (4,5,6,8) & (4,6,7) & (4,7,8)\\
 & & & (6,7,8) & & &

\end{tabular}.
\end{center}
By Proposition \ref{gen}, the arrangement $\mathcal{A}_{\mathfrak{C}}$ cannot be free. Using \texttt{SINGULAR} we can compute the degree sequence for $D_{0}(\mathcal{A}_{\mathfrak{C}})$, namely
$${\rm exp}_{0}(\mathcal{A}_{\mathfrak{C}}) = (4,4,4,4,4,4,4,4,4),$$
hence its type is equal to
$$t(\mathcal{A}_{\mathfrak{C}})= 4+4+4-8+1=5.$$

%================================================================================================================================
Based on the above discussion, we can formulate the main result of this section.
\begin{theorem}
\label{rigOC}
If $\mathcal{A} \subset \mathbb{P}^{3}$ is a rigid Cynk--Szemberg octic hyperplane arrangement enlisted in \cite{Cynks}, then $t(\mathcal{A}) \in \{1,2,3,4,5\}$.
\end{theorem}
\section{A few remarks about homological properties of one-parameter families of Cynk--Szemberg octics}
According to \cite[Section 4.2.5]{Meyer}, there are $63$ one-parameter families of Cynk--Szemberg octic hyperplane arrangements. It is natural to check whether these families possess interesting homological properties, especially with respect to the freeness properties. We examined general elements in these $63$ one-parameter families and concluded that \textit{there are no free nor nearly free examples}. There are $3$ families such that their general elements have resolution very similar to the nearly free case but the degree sequences do not satisfy properties that $d_{1}+d_{2}+d_{3}=8$ and $d_{3}=d_{4}$, which is somehow surprising.
Then we studied the degenerations of these families -- that is, the parameters for which the associated hyperplane arrangements have different intersection lattices. Let us present an example showing the meaning of \textit{degenerated arrangements}.
\begin{example}
We consider here family one-parameter family \textnumero 4 (see \cite[p. 23]{Cynks}), which is given by the following defining polynomial
\begin{equation}
Q_{A,B}(x,y,z,w)=xyzw(x+y)(y+z)(Ax+By+Bz-Aw)(Ax+Ay+Bz-Aw),  
\end{equation}
where $(A:B) \in \mathbb{P}^{1}$ subject to the conditions that $(A,B) \not\in\{(1:0),(0:1),(1,1)\}$.\\
For every admissible parameter $(A:B)$ the intersection poset of the arrangement has the following form:
\begin{center}
\begin{tabular}{ c c c c c  }
\multicolumn{5}{c}{$p$-fold lines} \\\hline \hline
 (1,2,5) & (1,3) &  (1,4) & (1,6) & (1,7) \\
 (1,8) & (2,3,6) & (2,4) & (2,7,8) & (3,4) \\
 (3,5) & (3,7) & (3,8) & (4,5) & (4,6) \\
 (4,7) & (4,8) & (5,6) & (5,7) & (5,8) \\
 & (6,7) &   &(6,8) &

\end{tabular}
\end{center}
and 
\begin{center}
\begin{tabular}{ c c c c c }
\multicolumn{5}{c}{arrangement $q$-fold points} \\ \hline \hline 
 (1,2,5) & (1,3) &  (1,4,7) &  (1,6) &  (1,8) \\
 (2,3,6) & (2,4) & (2,7,8) & (3,4) & (3,5) \\
 (3,7) & (3,8) & (4,5,8) & (4,6) &  (5,6) \\
 & (5,7) &   (6,7) &  (6,8)
\end{tabular}.
\end{center}
Let us now consider the arrangement $\mathcal{D}$ given by $Q_{1,0} = 0$. The intersection lattice of $\mathcal{D}$ has the following form:
\begin{center}
\begin{tabular}{ c c c c c  }
\multicolumn{5}{c}{$p$-fold lines} \\\hline \hline
 (1,2,5) & (1,3) & (1,4,7) & (1,6) & (1,8) \\
 (2,3,6) & (2,4) & (2,7,8) & (3,4) & (3,5) \\
 (3,7) & (3,8) & (4,5,8) & (4,6) & (5,6) \\
 & (5,7) & (6,7) & (6,8) &

\end{tabular}
\end{center}
and 
\begin{center}
\begin{tabular}{ c c c c c c }
\multicolumn{6}{c}{arrangement $q$-fold points} \\ \hline \hline 
 (1,2,3,5,6) & (1,2,4,5,7,8) & (1,3,4,7) & (1,3,8) & (1,4,6,7) & (1,6,8) \\
 (2,3,4,6) & (2,3,6,7,8) & (3,4,5,8) & (3,5,7) & (4,5,6,8) & (5,6,7)
\end{tabular}.
\end{center}
Obviously $\mathcal{D}$ is not a Cynk--Szemberg octic arrangement since $\mathcal{D}$ has a unique sixtuple intersection point, and this is the reason why we call this arrangement a degeneration. However, $$\pi(\mathcal{D},t) = 1+8t+23t^{2}+28t^{3}+12t^{4} = (1+t)(1+2t)^{2}(1+3t),$$
so $\mathcal{D}$ might be a free arrangement. Using \texttt{SINGULAR}, we can directly verify that $\mathcal{D}$ is indeed free with ${\rm exp}_{0}(\mathcal{D}) = (2,2,3)$.
\end{example}
\begin{remark}
We discovered more examples of free octic hyperplane arrangements using the idea of taking non-admissible parameters for these $63$ one-parameter families. However, the examples that we have found are rigid, so they cannot be used in the context of the Terao freeness conjecture.
\end{remark}
%--------------------------------------------------------------------------------------------------------------------------
\section*{Acknowledgments}
We would like to thank Alex Dimca for useful remarks and to S\l awomir Cynk for explaining us certain aspects of \cite{Cynks}.

Marek Janasz and Piotr Pokora are supported by the National Science Centre (Poland) Sonata Bis Grant  \textbf{2023/50/E/ST1/00025}. For the purpose of Open Access, the author has applied a CC-BY public copyright license to any Author Accepted Manuscript (AAM) version arising from this submission.

\vskip 0.5 cm
%*****************************Addresses*****************************
\bigskip

Marek Janasz, 
Department of Mathematics,
University of the National Education Commission Krakow,
Podchor\c a\.zych 2,
PL-30-084 Krak\'ow, Poland. \\
\nopagebreak
\textit{E-mail address:} \texttt{marek.janasz@uken.krakow.pl}
\bigskip

Piotr Pokora,
Department of Mathematics,
University of the National Education Commission Krakow,
Podchor\c a\.zych 2,
PL-30-084 Krak\'ow, Poland. \\
\nopagebreak
\textit{E-mail address:} \texttt{piotr.pokora@uken.krakow.pl}

\begin{thebibliography}{000}


\bibitem{her}
T. Abe, A. Dimca, and P. Pokora, A new hierarchy for complex plane curves. \textit{Canad. Math. Bull.}: 1 -- 24 (2025). \texttt{DOI:10.4153/S0008439525101422.}
\bibitem{Cynks}
S. Cynk and B. Kocel-Cynk, 
Classification of double octic Calabi-Yau threefolds with $h^{1,2}\leq 1$ defined by an arrangement of eight planes. \textit{Commun. Contemp. Math.} \textbf{22(1)}: Article ID 1850082, 38 p. (2020).

\bibitem{CSz}
S. Cynk and T. Szemberg, Double covers and Calabi-Yau varieties. Jakubczyk, Bronisław (ed.) et al., Singularities symposium – Łojasiewicz 70. Papers presented at the symposium on singularities on the occasion of the 70th birthday of Stanisław Łojasiewicz, Cracow, Poland, September 25–29, 1996 and the seminar on singularities and geometry, Warsaw, Poland, September 30–October 4, 1996. Warsaw: Polish Academy of Sciences, Institute of Mathematics, Banach Cent. Publ. 44, 93 -- 101 (1998).

\bibitem{Singular} W.~Decker, G.-M. Greuel, G.~Pfister, H.~Sch\"onemann, Singular {4-1-1} --- {A} computer algebra system for polynomial computations, \url{http://www.singular.uni-kl.de} (2018).
\bibitem{Dimca}
A. Dimca, \textit{Hyperplane arrangements. An introduction}. Universitext. Cham: Springer (ISBN 978-3-319-56220-9/pbk; 978-3-319-56221-6/ebook). xii, 200 p. (2017).

\bibitem{Dimca1}
A. Dimca, Some remarks on plane curves related to freeness. \textbf{arXiv:2501.01807}.

\bibitem{DS}
A. Dimca and G. Sticlaru, Free and nearly free surfaces in $\mathbb{P}^{3}$. \textit{Asian J. Math.} \textbf{22(5)}: 787 -- 810 (2018).


\bibitem{Meyer}
Ch. Meyer, \textit{Modular Calabi-Yau Threefolds}. Fields Institute Monographs 22. Providence, RI: American Mathematical Society (AMS) (ISBN 0-8218-3908-X/hbk). x, 195 p. (2005).


\bibitem{terao}
H. Terao, Generalized Exponents of a Free Arrangement of Hyperplanes and Shepherd-Todd-Brieskorn Formula. \textit{Invent. Math.} \textbf{63}: 159 -- 179 (1981).

\end{thebibliography}
\end{document}